\documentclass[12pt]{amsart}
\usepackage[pagewise]{lineno}
\usepackage{amscd,amssymb}
\usepackage[arrow,matrix]{xy}
\usepackage[plainpages,backref,urlcolor=blue]{hyperref}
\usepackage{mathrsfs}
\usepackage{amsthm}
\usepackage{bm}
\usepackage{amscd}
\usepackage[shortlabels]{enumitem}
\usepackage{mathrsfs}
\usepackage{graphics}
\usepackage{mathtools}

\theoremstyle{plain}
\newtheorem{thm}[subsection]{Theorem}
\newtheorem{lem}[subsection]{Lemma}
\newtheorem{prop}[subsection]{Proposition}
\newtheorem{cor}[subsection]{Corollary}
\newtheorem{claim}[subsection]{Claim}

\theoremstyle{definition}
\newtheorem{rk}[subsection]{Remark}
\newtheorem{definition}[subsection]{Definition}

\newcommand{\bb}{\mathbb}
\newcommand{\ov}{\overline}
\newcommand{\p}{\partial}

\newcommand{\Lin}{\text{\rm Lin}}

\begin{document}
	\date{}
		
	\title[On tangential deformations of homogeneous polynomials]{On tangential deformations of homogeneous polynomials}
		
	\author[ZHENJIAN WANG]{ ZHENJIAN WANG  }
	\address{Univ. Nice Sophia Antipolis, CNRS,  LJAD, UMR 7351, 06100 Nice, France.}
	\email{wzhj01@gmail.com}
		
		
		
	\begin{abstract}
    The Jacobian ideal provides the set of infinitesimally trivial deformations for a homogeneous polynomial, or for the corresponding complex projective hypersurface. In this article, we investigate whether the associated linear deformation  is indeed trivial, and show that the answer is no in a general situation. We also give a characterization of tangentially smoothable hypersurfaces with isolated singularities. Our results have applications in the local study of variations of projective hypersurfaces, complementing the global versions given by J.~Carlson and P.~Griffiths, R.~Donagi and the author, and in the study of isotrivial linear systems on the projective space, showing that a general divisor does not belong to an isotrivial linear system of positive dimension.
	\end{abstract}
	\maketitle
    \tableofcontents
	
\section{Introduction}
Let $S_n=\bb{C}[x_0,\cdots, x_n]$ be the graded ring of polynomials in $n+1$ variables $x_0,\cdots, x_n$ with coefficients in $\bb{C}$, which is also the homogeneous coordinate ring of $\bb{P}^n$, the $n$-dimensional complex projective space. $S_n$ admits a natural grading with respect to degree
$$
S_n=\bigoplus_{d=0}^\infty S_{n,d}
$$
where $S_{n,d}$ is the vector space of homogeneous polynomials of degree $d$. And any element $f\in S_{n,d}$ defines in $\bb{P}^n$ a hypersurface $H_f: f=0$, which is a projective scheme whose closed points are the zeros of $f$. We call $f$ {\bf smooth} if $H_f$ is a smooth hypersurface, {\bf singular} if otherwise.

The general linear group $G=GL(n+1,\bb{C})$ acts on $S_{n,d}$ by coordinate transformations. Given any nonzero $f\in S_{n,d}$, let $G\cdot f$ be the orbit of $f$ in $S_{n,d}$ and  $\bb{P}(G\cdot f)$ its image in $\bb{P}(S_{n,d})$ under projectivization; in addition, let $J_f$, called the {\bf Jacobian ideal of $f$}, be the graded ideal of $S_n$ generated by the partial derivatives of $f$:
$$
J_f=\biggl(\frac{\partial f}{\partial x_0}, \frac{\partial f}{\partial x_1},\cdots,\frac{\partial f}{\partial x_n}\biggr).
$$

A well-known fact states that the tangent space at $f$ to the orbit $G\cdot f$ is given by $T_f(G\cdot f)=J_{f,d}$, the degree $d$ homogeneous component of $J_f$.
In addition, from the viewpoint of deformation theory, $J_{f,d}$ exactly consists of all {\it infinitesimally trivial deformations} of $f$, see \cite{VO2}, Lemma 6.15.
 Moreover, as is shown in \cite{MM}, when $n\geq 3, d\geq 3$, then for a general $f$, except the case $(n,d)=(3,4)$, $\text{Aut}(H_f)$ is trivial and thus $H_f$ and $H_g$ are isomorphic as projective schemes if and only if they are projectively equivalent, i.e., $g=G\cdot f$. All these facts directly motivate the following definitions.

 Given two homogeneous polynomials $f,g\in S_{n,d}$, we say that $g$ is {\bf equivalent} to $f$, denoted by $g\cong f$, if $g\in G\cdot f$.

\begin{definition}
    Let $n\geq 1, d\geq 1$ and $f, h\in S_{n,d}$.
    \begin{enumerate}[(i)]
     \item If $h\in J_{f,d}$, $f_t=f+th,\ t\in\bb{C}$ is said to be a {\bf tangential deformation} for $f$. If, in addition, $f_t\cong f$ for all sufficiently small $t$ (i.e. $|t|<\epsilon$ for some $\epsilon>0$), $h$ is called a {\bf tangentially trivial} deformation for $f$.
	\item $f$ is called {\bf totally tangentially unstable} if any tangentially trivial deformation for $f$ is a complex multiple of $f$.
    \end{enumerate}
\end{definition}

We shall prove the following.

\begin{thm}\label{thm: TTI}
For $n\geq 1$ and $d\geq 4$, a {\bf general} $f\in S_{n,d}$ is totally tangentially unstable.
\end{thm}

For the deformation of a singular polynomial, we can also consider another property.

\begin{definition}\label{def: TS}
	Let $n\geq 2, d\geq 2$ and $f\in S_{n,d}$ be singular. $f$ is said to be {\bf tangentially smoothable} if there exists an $h\in J_{f,d}$ such that $f+h$ is smooth.
\end{definition}

We prove the following.

\begin{thm}\label{thm: TS}
	 Let $n\geq 2, d\geq 2$ and $f\in S_{n,d}$. If $f$ is singular and $H_f$ has only isolated singularities, then $f$ is tangentially smoothable if and only if every singular point of $H_f$ has multiplicity 2.
\end{thm}

As a corollary, we have

\begin{cor}
   A {\bf general} singular polynomial $f\in S_{n,d}$ is tangentially smoothable.
\end{cor}

Our results have applications in the local study of variations of projective hypersurfaces, complementing the global versions given by J.~Carlson and Griffiths \cite{CG}, R.~Donagi \cite{DO} and the author \cite{ZW}, and in the study of isotrivial linear systems on the projective space, showing that a general divisor does not belong to an isotrivial linear system of positive dimension.\\

The author would like to thank an anonymous referee, who pointed out the relations between tangential deformations and isotrivial pencils, and independently gave a new approach to prove almost all the main results of this paper. Following the referee's remarks, the author put all the methods together and thus made an improvement of the previous version of this paper.

\section{Basic properties of totally tangential instability}\label{sec: TTIproperties}

Let $n\geq 1$ and $ d\geq 1$. We fix an $f\in S_{n,d}$. Define $\mathcal{T}_f$ to be the set of tangentially trivial deformations of $f$:
\begin{equation}\label{eq: mathcalTf}
\mathcal{T}_f=\{\ h\in J_{f,d}\quad :\quad f+th\cong f\text{ for small }t\  \}.
\end{equation}

\begin{lem}\label{lem: cone}
	For $n\geq 1, d\geq 1$ and any $f\in S_{n,d}$. Then the following hold:
\begin{enumerate}[{\rm (i)}]
\item $\mathcal{T}_f$ is a cone in $S_{n,d}$, i.e., $\lambda h\in\mathcal{T}_f$ for any $h\in\mathcal{T}_f$ and $\lambda\in\bb{C}$.
\item Furthermore, $\bb{T}_f:=\bb{P}(\mathcal{T}_f)$ is also a cone with apex corresponding to the polynomial $f$, i.e., for any $h\in\bb{T}_f$ and $\lambda\in\bb{C}$, we have $f+\lambda h\in\bb{T}_f$.
\end{enumerate}
\end{lem}

\begin{proof}
Indeed, for any $h\in\mathcal{T}_f$, $f+th\cong f$ for small $t$ by definition. Given $\lambda\in\bb{C}$, let $t'=\lambda t$, then $f+t(\lambda h)=f+ t'h\cong f$ for small $t$, i.e., $\lambda h\in \mathcal{T}_f$.

Now observe that
$$
f+t(f+\lambda h)=(1+t)\biggl(f+\frac{\lambda t}{1+t}h\biggr)
$$
so if $h\in\bb{T}_f$, by definition,
$$
f+\frac{\lambda t}{1+t}h\cong f
$$
for $t$ small, so $f+t(f+\lambda h)\cong f$ for $t$ small, i.e., $f+\lambda h\in \mathcal{T}_f$.\qedhere
\end{proof}

Observe that $f$ is totally tangentially unstable if and only if $\bb{T}_f=\{f\}$.

An obvious corollary follows from Lemma \ref{lem: cone}.

\begin{cor}\label{cor: connected}
	For $n\geq 1, d\geq 1$ and $f\in S_{n,d}$, $\bb{T}_f\subseteq\bb{P}(S_{n,d})$ is connected in the strong topology.
\end{cor}

\begin{rk}
\begin{enumerate}
\item Since we did not prove the algebraic nature of $\bb{T}_f$, we cannot say at the moment that it is connected in the Zariski topology. Actually, $\bb{T}_f$ is constructible (even Zariski closed) for any $f$, so it also holds the connectedness in the Zariski topology, because connectedness in the strong topology and that in the Zariski topology are the same.
\item If $k\in S_{n,d}$ satisfies $f+tk\cong f$ for small $t$, then necessarily, $k\in J_{f,d}$ since $J_{f,d}$ is the tangent space of the orbit $G\cdot f$ at $f$. So we get the following characterization of $\mathcal{T}_f$:
    $$
    \mathcal{T}_f=\{\ h\in S_{n,d}\quad :\quad f+th\cong f\text{ for small }t\  \}.
    $$
\end{enumerate}
\end{rk}

Now let $M_{(n+1)\times(n+1)}$ be the vector space of $(n+1)\times(n+1)$ matrices. Fix an $f\in S_{n,d}$, and set
$$
\mathfrak{A}_f=\{\ f\circ A\quad:\quad A\in M_{(n+1)\times(n+1)}\ \}.
$$
Then $\mathfrak{A}_f$ is an irreducible constructible subset of $S_{n,d}$.  We have the following useful characterization for $\mathcal{T}_f$.

\begin{prop}\label{prop: charMTf}
With the notations as above, we have
$$
\mathcal{T}_f=\{\ h\in S_{n,d}\quad:\quad f+th\in\mathfrak{A}_f\text{ for small } t\ \}.
$$
\end{prop}

\begin{proof}
If suffices to show $(G\cdot f, f)=(\mathfrak{A}_f,f)$ as an equality of reduced germs at $f$. Indeed, as $G=GL(n+1,\bb{C})$ is a Zariski open subset of $M_{(n+1)\times(n+1)}$, it follows that $G\cdot f$ contains a Zariski open subset of $\mathfrak{A}_f$. Moreover,  as an orbit of group action, $G\cdot f$ is a smooth variety, so $f$ is an interior point of $G\cdot f$. It follows that $G\cdot f$ and $\mathfrak{A}_f$ coincide in a small neighbourhood of $f$. We are done.
\end{proof}

\section{Openness of totally tangential instability}

Let $\mathcal{U}_{n,d}\subseteq S_{n,d}$ be the set of totally tangentially unstable polynomials, namely,
$$
\mathcal{U}_{n,d}=\{\ f\in S_{n,d}\quad:\quad f\text{ is totally tangentially unstable}\ \}.
$$
Then $\mathcal{U}_{n,d}$ is a cone, so we can consider its projectivization $\bb{P}(\mathcal{U}_{n,d})$.

\begin{prop}\label{prop: TTIopenness}
Let $n\geq1$ and $d\geq 1$. Then $\bb{P}(\mathcal{U}_{n,d})$ is a Zariski open subset of $\bb{P}(S_{n,d})$.
\end{prop}
\begin{proof}
We first give the outline of the proof containing several claims, then we prove our claims.

{\bf Step 1: Outline of the proof }

Consider the incidence variety
$$
\mathfrak{X}=\{\ (f,h)\in \bb{P}(S_{n,d})\times \bb{P}(S_{n,d})\quad : \quad h\in\bb{T}_f\ \}.
$$
Then we have
\begin{claim}\label{claim: X-closed}
$\mathfrak{X}$ is a Zariski closed subset in $\bb{P}(S_{n,d})\times \bb{P}(S_{n,d})$.
\end{claim}

Let $\pi:\mathfrak{X}\to\bb{P}(S_{n,d})$ be the projection to the first factor. Then $\pi$ is surjective since always $f\in\bb{T}_f$. Moreover, the following holds:
\begin{claim}\label{claim: dim fiber}
We have
$$
\bb{P}(\mathcal{U}_{n,d})=\{\ f\in\bb{P}(S_{n,d})\quad:\quad \dim\pi^{-1}(f)=0\ \}.
$$
\end{claim}

Set
$$
\mathfrak{C}=\{\ f\in\bb{P}(S_{n,d})\quad:\quad \dim\bb{P}(\pi^{-1}(f))\geq 1\ \},
$$
then $\mathfrak{C}$ is a Zariski closed subset of $\bb{P}(S_{n,d})$ by the Semi-continuity Theorem of Chevalley (see Section 4.5, \cite{DS}). Hence, $\bb{P}(\mathcal{U}_{n,d})=\bb{P}(S_{n,d})\setminus\mathfrak{C}$ is Zariski open.\\

{\bf Step 2: Proof of Claim \ref{claim: X-closed} }

Denote $W=(S_{n,d}\setminus\{0\})\times (S_{n,d}\setminus\{0\})$. Consider the incidence variety
$$
\mathfrak{Y}=\{\ ((f,h),t,A)\in W\times\bb{C}\times M_{(n+1)\times (n+1)}\quad : \quad (f+th)(X)=f(AX)\ \}
$$
where $X=(x_0,\cdots, x_n)^T$ is the vector of variables and $AX$ is obtained by matrix multiplication.
Then $\mathfrak{Y}$ is Zariski closed in $W\times\bb{C}\times M_{(n+1)\times (n+1)}$. Furthermore, let
$$
\mathfrak{Y}'=W\times\bb{C}
$$
and $p:\mathfrak{Y}\to\mathfrak{Y}'$ be the natural projection. Observe that $\mathfrak{Y},\mathfrak{Y}'$ can be regarded as $W$-schemes and $p$ as a $W$-morphism. In addition, for any $w\in W$, we get an induced morphism on the fiber $p_w: \mathfrak{Y}_w\to\mathfrak{Y}'_w$. Finally, let
$$
\mathfrak{B}=\{\ w\in W\quad :\quad p_w \text{ is dominant }\}
$$
then $\mathfrak{B}$ is constructible, see \cite{EGAIV}(9.6.1).

We claim  that $\mathfrak{B}$ is Zariski closed in $W$. Because $\mathfrak{B}$ is constructible, it suffices to prove that $\mathfrak{B}$ is closed in $W$ in the strong topology by \cite{Mu}, Chapter I, \S 10, Corollary 1.
To this end, let $w_i=(f_i,h_i), i=1,2,\cdots$ be a sequence in $\mathfrak{B}$ such that $w_i\to w_\infty=(f_\infty,h_\infty)\in W$ in the strong topology. The problem is reduced to prove $p_{w_\infty}(\mathfrak{Y}_{w_\infty})$ is dense in $\bb{C}$ in the strong topology, since $p_{w_\infty}(\mathfrak{Y}_{w_\infty})$ is constructible in the Zariski topology by Chevalley's constructibility theorem (see also \cite{Mu}, Chapter I, \S 8, Corollary 3).

Let $U_i=p_{w_i}(\mathfrak{Y}_{w_i})$ for $i\geq 1$, then $U_i$ is a Zariski open dense subset of $\bb{C}$, hence $\bb{C}\setminus U_i$ is closed and nowhere dense in $\bb{C}$ in the strong topology. Set
$$
U=\bigcap_{i=1}^\infty U_i,
$$
then by Baire Category Theorem, $U$ is dense in $\bb{C}$ in the strong topology since $\bb{C}$ is a complete metric space. Now for any $t\in U$ and $t\neq 0$, we will show $t\in p_{w_\infty}(\mathfrak{Y}_{w_\infty})$. Indeed, since $(f_i,h_i)\in\mathfrak{B}$, there exists a $B_i\in M_{(n+1)\times(n+1)}$ for each $i$, such that $(f_i+th_i)(X)=f_i(B_iX)$. We may assume $f_\infty+th_\infty\neq 0$ in $S_{n,d}$ since clearly $(-th,h)\in\mathfrak{B}$ for any $h\neq 0$ (note that we assumed $t\neq 0$). In particular, $B_i\neq0$ for $i$ large enough. By the compactness of $\bb{P}(M_{(n+1)\times(n+1)})$, we may assume further $B_i\to B_\infty$ in $\bb{P}(M_{(n+1)\times(n+1)})$. So letting $i\to\infty$, we get $(f_\infty+th_\infty)(X)=f_\infty(B_\infty X)$ in $\bb{P}(S_{n,d})$, implying that $t\in p_{w_\infty}(\mathfrak{Y}_{w_\infty})$ as desired. It follows that $U\subseteq p_{w_\infty}(\mathfrak{Y}_{w_\infty})$ and hence our claim about Zariski closeness of $\mathfrak{B}$ follows.

Let
$$
q\quad:\quad W=(S_{n,d}\setminus\{0\})\times(S_{n,d}\setminus\{0\})\to\bb{P}(S_{n,d})\times\bb{P}(S_{n,d})
$$
be the natural quotient morphism and let $\mathfrak{X}'=q(\mathfrak{B})$. Then $\mathfrak{X}'$ is a Zariski closed subset of $\bb{P}(S_{n,d})\times\bb{P}(S_{n,d})$. In fact, $q$ is a topological quotient map and as can be easily shown, $\mathfrak{B}=q^{-1}(q(\mathfrak{B}))$, we have from the closeness of $\mathfrak{B}$ that $\mathfrak{X'}=q(\mathfrak{B})$ is closed in the strong topology. Thus $\mathfrak{X}'$ is Zariski closed since it is constructible in the Zariski topology.

Now, to finish the proof of Claim \ref{claim: X-closed}, it is enough to show that $\mathfrak{X}=\mathfrak{X}'$.

Indeed, if $w=(f,h)\in \mathfrak{X}$, then necessarily, $p_w(\mathfrak{Y}_w)$ at least contains a small neighborhood of 0, hence a Zariski open dense subset of  $\bb{C}$. Thus $\mathfrak{X}\subseteq\mathfrak{X}'$.

Conversely, for any $w=(f,h)\in \mathfrak{X}'$, $p_w$ contains a Zariski open dense subset of $\bb{C}$. Then, for any $t\in\bb{C}$ small, we can find $A\in M_{(n+1)\times (n+1)}$ such that $(f+th)(X)=f(AX)$. By Proposition \ref{prop: charMTf}, we get $h\in\bb{T}_f$. Hence, $\mathfrak{X}'\subseteq\mathfrak{X}$.\\

{\bf Step 3: Proof of Claim \ref{claim: dim fiber} }
 Denote
 $$
 \mathfrak{D}_0=\{\ f\in{P}(S_{n,d})\quad:\quad \dim\pi^{-1}(f)=0\ \},
 $$
 and we shall prove $\bb{P}(\mathcal{U}_{n,d})=\mathfrak{D}_0$.

 Note first that $\pi^{-1}(f)=\bb{T}_f$.  The inclusion $\bb{P}(\mathcal{U}_{n,d})\subseteq \mathfrak{D}_0$ follows immediately from the definition. To show the reverse inclusion, choose $f\in\mathfrak{D}_0$, then $\dim\bb{T}_f=\dim\pi^{-1}(f)=0$. By Corollary \ref{cor: connected}, $\bb{T}_f$ is connected. It follows that $\bb{T}_f$ consists of only one point, namely, $\bb{T}_f=\{f\}$ since always $f\in\bb{T}_f$. Hence $f$ is totally tangentially unstable and the desired inclusion follows.\qedhere
\end{proof}

As a corollary of the proof above, we have the following.

\begin{cor}
Let $n\geq1$ and $d\geq 1$. For any $f\in \bb{P}(S_{n,d})$, $\bb{T}_f$ is a Zariski closed subset of $\bb{P}(S_{n,d})$.
\end{cor}

\section{Genericity of totally tangential instability}

In this section, we prove that $\bb{P}(\mathcal{U}_{n,d})$ is nonempty for $n\geq1$ and $d\geq 4$. As $\bb{P}(\mathcal{U}_{n,d})$ is Zariski open by Proposition \ref{prop: TTIopenness}, we conclude the proof of Theorem \ref{thm: TTI}.

We first reduce the problem to the case $n=1$.

For convenience, here we fix the notation for a pencil of two polynomials. Let $n\geq 1$ and $d\geq 1$. Given two polynomials $F,G\in S_{n,d}$, we denote by
\begin{equation}\label{eq: mathcal PFG}
\mathcal{P}_{F,G}=\{\ \lambda F+\mu G\quad : \quad (\lambda:\mu)\in\bb{P}^1\ \}
\end{equation}
the pencil determined by $F$ and $G$. We shall call an element in $\mathcal{P}_{F,G}$ a \emph{fiber} of it.

\begin{prop}\label{prop: TTInonempty}
Let $n\geq 1$ and $d\geq 2$. If $\bb{P}(\mathcal{U}_{n,d})$ is nonempty, then $\bb{P}(\mathcal{U}_{n+1,d})$ is nonempty.
\end{prop}

\begin{proof}
We prove under the assumption, a {\bf general} $f\in S_{n+1,d}$ belongs to $\mathcal{U}_{n+1,d}$. To this end, let $h\in\mathcal{T}_f\subseteq S_{n+1,d}$.

If $h\in\bb{C}f$, we are done. Thus we assume that $h\notin \bb{C}f$ and will derive a contradiction.

Let $\bb{P}^{n+1*}$ be the dual projective space, i.e., the variety parameterizing all hyperplanes in $\bb{P}^{n+1}$. Then for a general element $L:\ell=0$ in $\bb{P}^{n+1*}$, we have $L\cong\bb{P}^n$ and the restriction of $f$ on $L$, denoted by $f|_L$, can be seen as a general element in $S_{n,d}$. By assumption, $\bb{P}(\mathcal{U}_{n,d})$ is nonempty and it is Zariski open by Proposition \ref{prop: TTIopenness}, so $f|_L\in\mathcal{U}_{n,d}$. Moreover, we have $h|_L\in\mathcal{T}_{f|_L}$. It follows that $h_L$ is a multiple of $f|_L$; in other words, there exists a constant $\lambda_L\in\bb{C}$ such that $(h-\lambda_Lf)$ is divided by $\ell$.

We claim that $\lambda_L$ admits infinitely many values as $L$ varies in a Zariski open subset in $\bb{P}^{n+1*}$. Indeed, if not, some $\lambda_0\in\bb{C}$ would be obtained as $\lambda_L$ for infinitely many $L$. Then $h-\lambda_0 f$ would have infinitely many linear factors; this contradicts our assumption that $h\notin\bb{C}f$.

Note that for any general $L$, $h-\lambda_L f$ has a linear factor $\ell$; a fortiori, $h-\lambda_L f$ is reducible. Hence, the pencil $\mathcal{P}_{f,h}$ contains an infinite number of reducible fibers. But this is impossible: note that $H_f$ is irreducible as $f$ is generically chosen. It follows that a general fiber in the pencil $\mathcal{P}_{f,h}$ is irreducible, implying that the number of reducible fibers in the pencil $\mathcal{P}_{f,h}$ is finite. More precisely, the number of reducible fibers in the pencil is $\leq d^2-1$ by the theorem in \cite{VI}.\qedhere
\end{proof}

We shall concentrate on the nonemptiness of $\mathcal{U}_{1,d}$ for $d\geq4$.

Given $D$ a divisor on $\bb{P}^n$. Denote by $\Lin(D)$ the subgroup of $\text{Aut}(D)$ consisting of elements induced by projective transformations in $\text{Aut}(\bb{P}^n)$ which leave $D$ invariant.

\begin{thm}\label{thm: TTIn=1}
Let $f$ be a smooth binary form in $x,y$ of degree $d$ and $H_f: f=0$ be the associated divisor in $\bb{P}^1$. If $\Lin(H_f)=\{\text{\rm Id}\}$, then $f$ is totally tangentially unstable, i.e., $f\in\mathcal{U}_{1,d}$.
\end{thm}

\begin{proof}
Let $h\in\mathcal{T}_f$ and aiming at a contradiction, we assume $h\notin\bb{C}f$. Then for a general $(\lambda:\mu)\in\bb{P}^1$, the divisor $H_{\lambda f+\mu h}$ is projectively equivalent to $H_f$, namely, there exists $\phi_{(\lambda:\mu)}\in\text{Aut}(\bb{P}^1)$ such that $\phi_{(\lambda:\mu)}^*H_{\lambda f+\mu h}=H_f$. Note that $\phi_{(\lambda:\mu)}$ is uniquely given since $\Lin(H_f)=\{\text{Id}\}$ and in addition, $\phi_{(\lambda:\mu)}$ is not constant as a function of $(\lambda:\mu)$ since $h\notin\bb{C}f$. Moreover, $\phi_{(\lambda:\mu)}$, as a function of $(\lambda:\mu)$, is rational since so is the divisor $H_{\lambda f+\mu h}$ seen as a function from $\bb{P}^1$ to the space of degree $d$ divisors on $\bb{P}^1$.

Now let $a=\gcd(f,h)$ and $f=ab, h=ac$. Then $H_a$ is the fixed part (or base locus) for the pencil $\{H_{\lambda f+\mu h}\}$ and $\{H_{\lambda b+\mu c}\}$ gives the moving part. For a general $(\lambda_0:\mu_0)$, $\{\phi_{(\lambda_0:\mu_0)}^*H_{\lambda f+\mu h}\}$ is a pencil of divisors with fixed part $\phi_{(\lambda_0:\mu_0)}^*H_a$.

We claim that $\phi_{(\lambda_0:\mu_0)}^*H_a=H_a$, namely, $H_a$ is preserved by $\phi_{(\lambda_0:\mu_0)}$. Indeed, it suffices to prove the claim for $(\lambda_0:\mu_0)$ lying in a small neighbourhood of $(1:0)$ in the strong topology. Notice also that $\phi_{(1:0)}=\text{Id}$, so by continuity, $\phi_{(\lambda_0:\mu_0)}$ is close to the identity when $(\lambda_0:\mu_0)$ is close to $(1:0)$. It follows that the support of $\phi_{(\lambda_0:\mu_0)}^*H_a$ is contained in a small neighbourhood of the support of $H_a$. Meanwhile, notice that $\phi_{(\lambda_0:\mu_0)}^*H_a$ is a sub-divisor of $\phi_{(\lambda_0:\mu_0)}^*H_{\lambda_0f+\mu_0h}=H_f$ and the support of $H_f=H_a+H_b$ consists of $d$ distinct points. Therefore, the support of $\phi_{(\lambda_0:\mu_0)}^*H_a$ must be contained in that of $H_a$. Since $\phi_{(\lambda_0:\mu_0)}$ is an automorphism of $\bb{P}^1$, we deduce that $\phi_{(\lambda_0:\mu_0)}^*H_a=H_a$.

Notice that $(\lambda_0:\mu_0)$ can be generically chosen, so $H_a$ is, a fortiori, preserved by an infinite number of automorphisms of $\bb{P}^1$. It follows that $H_a$ is supported on at most two points, and hence $\deg a\leq 2$ since any point in $H_a$ has multiplicity one by our assumption that $f$ is smooth.

We claim that $\deg b=\deg c=1$. Indeed, consider the following rational map
$$
\Phi\quad:\quad\bb{P}^1\times\bb{P}^1\dashrightarrow\bb{P}^1\times\bb{P}^1
$$
$$
    ((\lambda:\mu),(x:y))\mapsto((\lambda:\mu),\phi_{(\lambda:\mu)}(x:y)).
$$
Since $\phi_{(\lambda:\mu)}$ is an automorphism of $\bb{P}^1$, it follows that $\Phi$ is birational. Moreover, recall that $\lambda f+\mu h=a(\lambda a+\mu c)$ and $\phi_{(\lambda:\mu)}^*(\lambda b+\mu c)=b$ which is independent of $(\lambda:\mu)$, so $\Phi$ maps the curve defined by $\lambda b+\mu c$ into $\deg b$ disjoint union of horizontal curves.  Note that $\lambda b+\mu c\in\bb{C}[\lambda,\mu]\otimes\bb{C}[x,y]$ is an irreducible polynomial and thus the image of the curve defined by $\lambda b+\mu c$ under $\Phi$ is also irreducible. It follows that $\deg b=1$ as desired.

Putting everything together, we have that $d=\deg f=\deg a+\deg b\leq3$, and hence $\Lin(H_f)$ contains more than one element, contradicting our assumption.
\end{proof}

Notice that for a general $f\in S_{1,d}$ for $d\geq 4$, $\Lin(H_f)=\{\text{Id}\}$, hence combined with Proposition \ref{prop: TTInonempty} and Proposition \ref{prop: TTIopenness}, the above result gives the following, proving Theorem \ref{thm: TTI}.

\begin{cor}
For $n\geq 1$ and $d\geq 4$, $\bb{P}(\mathcal{U}_{n,d})$ is a nonempty Zariski open subset of $\bb{P}(S_{n,d})$.
\end{cor}

\begin{rk}
For $n=1$, the result is sharp: in fact, for any smooth $f\in S_{1,3}$, the orbit $G\cdot f$ is a Zariski open subset of $S_{1,3}$. It follows that a general $f\notin\mathcal{U}_{1,3}$ and $\bb{P}(\mathcal{U}_{1,3})=\emptyset$.

For $n=2$, the result is also sharp. Indeed, a general cubic curve in $\bb{P}^2$ is projectively equivalent to a cubic $C$ in Weierstrass form, i.e.,
$$
C=\{zy^2=x^3+az^2x+bz^3\}
$$
for suitable $a,b\in\bb{C}$. Let $f=zy^2-(x^3+az^2x+bz^3)$ and $h=zy^2$, then
$$
f+th=z(y\sqrt{1+t})^2-(x^3+az^2x+bz^3)\in G\cdot f
$$
for any small $t$, hence $h\in\mathcal{T}_f\setminus\bb{C}f$. So $f\notin\mathcal{U}_{2,3}$.

We conjecture that the result is no longer sharp for $n\geq3$, or in other words, $\bb{P}(\mathcal{U}_{n,3})\neq\emptyset$. However, we have not managed to prove this. Perhaps it is helpful to notice that the group $\text{Aut}(H_f)$ is trivial by the theorem in \cite{MM} for a general $f\in S_{n,3}$ when $n\geq 3$, which is not true when $n=1,2$.
\end{rk}

\section{Constructions of totally tangential unstable polynomials}

By the main theorem \ref{thm: TTI}, there are plenty of totally tangentially unstable homogeneous polynomials of degree $d\geq 4$. However, it is still not easy to explicitly write down such a polynomial except $n=1$. In this section, we will provide a procedure for the construction of a totally tangentially unstable polynomial, and thus give an alternative proof of the nonemptiness of $\bb{P}(\mathcal{U}_{n,d})$ for $d\geq 7$.

To this end, we first prove some basic results on pencils of polynomials, which are also interesting for their own right.

\begin{prop}\label{prop: connected fibers}
Let $n\geq 2$ and $d\geq 3$. Suppose $f,g\in S_{n,d}$ such that $f$ is decomposed as $f=p_1p_2$ where $p_i, i=1,2$ are irreducible, coprime and $\deg p_1\neq \deg p_2$, and $\gcd(f,g)=1$.

Then a general fiber of the map
$$
\rho: M=\bb{P}^n\setminus H_{fg}\longrightarrow C,\qquad \rho=(f:g)
$$
is connected, where $C=\bb{P}^1\setminus B$ with $B=\{(1:0),(0:1)\}$.
\end{prop}

\begin{proof}
Consider the Stein factorization: there exists a finite map $p_0: C_0\rightarrow C$ and a morphism $\rho_0: M\rightarrow C_0$ such that a general fiber of $\rho_0$ is connected, and $\rho=p_0\circ\rho_0$. Note that $C_0$ is also a noncompact curve, so it has the form $C_0=\ov{C_0}\setminus B_0$ where $\ov{C_0}$ is complete smooth curve and $B_0\subseteq\ov{C_0}$ is a finite set of points.

In the sequel, we first give an outline of the proof containing some claims, then we prove our claims.

{\bf Step 1: Outline}

First, we have
\begin{claim}\label{claim: p1}
$\ov{C}_0=\bb{P}^1$.
\end{claim}
So $\rho_0:M\rightarrow C_0$ is given by $\rho_0=(P_1:P_2)$ for two coprime homogeneous polynomials $P_1, P_2$. Moreover, $p_0:C_0\subseteq\bb{P}^1\to C\subseteq\bb{P}^1$ extends to a morphism $p_0':\ov{C_0}=\bb{P}^1\to\bb{P}^1$, which is necessarily of the form
$$
p_0'(u,v)=\biggl(\prod_{j=1}^m(a_{1,j}u+b_{1,j}v):\prod_{j=1}^m(a_{2,j}u+b_{2,j}v)\biggr)
$$
where $(a_{1,j}:b_{1,j})\neq(a_{2,k}: b_{2,k}),\  j,k=1,\cdots, m$ in $\bb{P}^1$.

If $m=1$, then $p_0'$ is an isomorphism and we are done. Thus suppose $m>1$. Then it follows from $\rho=p_0\circ\rho_0$ that
$$
(f:g)=\biggl(\prod_{j=1}^m(a_{1,j}P_1+b_{1,j}P_2):\prod_{j=1}^m(a_{2,j}P_1+b_{2,j}P_2)\biggr).
$$
and furthermore, we deduce that
\begin{claim}\label{claim: fg}
$$
\begin{cases}
f=c\prod_{j=1}^m(a_{1,j}P_1+b_{1,j}P_2)\\
g=c\prod_{j=1}^m(a_{2,j}P_1+b_{2,j}P_2)
\end{cases}
$$
for some $c\in\bb{C}$.
\end{claim}

By assumption, $m>1$ and $f$ has exactly two irreducible factors $p_1,p_2$. It follows that $m=2$ and there exist complex numbers $c_{i,j}, i=1,2, j=1,2$ such that $p_i=c_{i,1} P_1+c_{i,2} P_2$ for $i=1,2$. But then $\deg p_1=\deg p_2$, contradicting our assumption that $\deg p_1\neq \deg p_2$.\\

{\bf Step 2: Proof of Claim \ref{claim: p1}: }

Since $\rho_0: M\rightarrow C_0$ has a connected general fiber, the induced map between fundamental groups $\rho_{0,*}:\pi_1(M)\to\pi_1(C_0)$ is surjective, hence $\rho_0^*: H^1(C_0)\to H^1(M)$ is injective. Note that $\rho_0^*$ is a morphism of mixed Hodge structures, so
$$
\rho_0^*: W_1H^1(C_0)\to W_1H^1(M)
$$
is also injective.

Now we have $W_1H^1(C_0)=\text{Image}(H^1(\ov{C_0})\to H^1(C_0))\cong H^1(\ov{C_0})$ and
$$
W_1H^1(M)=\text{Image}(H^1(\bb{P}^n)\to H^1(M))\cong H^1(\bb{P}^n)=0,
$$
therefore $H^1(\ov{C_0})=0$ and thus $\ov{C_0}\cong\bb{P}^1$.\\

 {\bf Step 3: Proof of Claim \ref{claim: fg} }

The main point is to show
$$\prod_{j=1}^m(a_{1,j}P_1+b_{1,j}P_2)\quad \text{and}\quad \prod_{j=1}^m(a_{2,j}P_1+b_{2,j}P_2)$$ are coprime.
First, the polynomials $a_{1,j}P_1+b_{1,j}P_2, a_{2,k}P_1+b_{2,k}P_2$ are mutually coprime for any $j,k=1,\cdots,m$. Indeed, $\gcd(a_{1,j}P_1+b_{1,j}P_2,a_{2,k}P_1+b_{2,k}P_2)=\gcd(P_1,P_2)=1$ since $(a_{1,j}: b_{1,j})\neq (a_{2,k}:b_{2,k})$ in $\bb{P}^1$. Hence
$$
\gcd\biggl(\prod_{j=1}^m(a_{1,j}P_1+b_{1,j}P_2),\prod_{j=1}^m(a_{2,j}P_1+b_{2,j}P_2)\biggr)=1.
$$

Now, from the equality
$$
(f:g)=\biggl(\prod_{j=1}^m(a_{1,j}P_1+b_{1,j}P_2):\prod_{j=1}^m(a_{2,j}P_1+b_{2,j}P_2)\biggr)
$$
we get that $g\cdot \prod_{j=1}^m(a_{1,j}P_1+b_{1,j}P_2)-f\cdot\prod_{j=1}^m(a_{2,j}P_1+b_{2,j}P_2)$ vanishing identically on $M=\bb{P}^n\setminus H_{fg}$, thus for some $s$ sufficiently large,
$$
(fg)^s\biggl(g\cdot \prod_{j=1}^m(a_{1,j}P_1+b_{1,j}P_2)-f\cdot\prod_{j=1}^m(a_{2,j}P_1+b_{2,j}P_2)\biggr)
$$
vanishes identically on $\bb{P}^n$ (see \cite{Ha}, Lemma 5.14), hence it is zero as a polynomial, and thus we obtain an equality of polynomials in $\bb{C}[x_0,\cdots, x_n]$:
$$
g\cdot \prod_{j=1}^m(a_{1,j}P_1+b_{1,j}P_2)=f\cdot\prod_{j=1}^m(a_{2,j}P_1+b_{2,j}P_2).
$$

Further,
$$
\gcd(f,g)=1,\qquad\gcd\biggl(\prod_{j=1}^m(a_{1,j}P_1+b_{1,j}P_2),\prod_{j=1}^m(a_{2,j}P_1+b_{2,j}P_2\biggr)=1,
$$
we obtain that
$$
f\bigg|\prod_{j=1}^m(a_{1,j}P_1+b_{1,j}P_2),\quad\text{and}\quad \prod_{j=1}^m(a_{1,j}P_1+b_{1,j}P_2)\bigg|f,
$$
so $f$ is a multiple of $\prod_{j=1}^m(a_{1,j}P_1+b_{1,j}P_2)$; similar for $g$.\qedhere
\end{proof}

\begin{cor}\label{cor: red+finite}
Under the assumption of Proposition \ref{prop: connected fibers}, the pencil $\mathcal{P}_{f,g}$ contains at most $d^2-1$ reducible fibers.
\end{cor}
\begin{proof}
This follows immediately from the main theorem in \cite{VI}, if we show that a general fiber $F$ in $\mathcal{P}_{f,g}$ is irreducible.

Indeed, since $\gcd(f,g)=1$, the base locus $H_f\cap H_g$ of the pencil $\mathcal{P}_{f,g}$ has codimension 2 in $\bb{P}^n$. Notice that $H_F$ is pure of codimension 1. So it suffices to prove $H_F\setminus(H_f\cap H_g)$ is irreducible. Indeed, it is connected by Proposition  \ref{prop: connected fibers}, and  smooth by Bertini Theorem, therefore it is irreducible as desired.
\end{proof}

Now we begin to construct totally tangentially unstable polynomials.

\begin{prop}\label{prop: constructionp1p2}
Let $n\geq 2$ and $d\geq 7, d_1\geq 3, d-d_1\geq 3, d_1\neq d-d_1$. Let $p_1\in S_{n,d_1}, p_2\in S_{n,d-d_1}$ be two smooth polynomials, then $f=p_1p_2$ is totally tangentially unstable.
\end{prop}

\begin{proof}
For any $h\in\mathcal{T}_f$, we have $f+th\in G\cdot f$ for $t$ small. It follows that $f+th$ is reducible for $t$ small.

Let $g=f+t_0h$ for $t_0\neq 0$ small enough. Then the pencil $\mathcal{P}_{f,g}$ contains infinitely many reducible fibers, hence by Corollary \ref{cor: red+finite}, we get that $\gcd(f,g)\neq 1$. So $p_1|g$ or $p_2|g$. Without loss of generality, we assume $p_1|g$.

Moreover, note that $g=f+t_0h\in J_{f,d}$, so there exist $a_{i,j}\in\bb{C}, i,j=0,\cdots, n$ such that
$$
g=\sum_{i,j=0}^na_{i,j}x_i\frac{\p f}{\p x_j}
$$
or in other words,
$$
g=p_2\cdot\sum_{i,j=0}^na_{i,j}x_i\frac{\p p_1}{\p x_j}+p_1\cdot\sum_{i,j=0}^na_{i,j}x_i\frac{\p p_2}{\p x_j}.
$$
Since $p_1|g$ and $\gcd(p_1,p_2)=1$, we have
$$
p_1\bigg|\sum_{i,j=0}^na_{i,j}x_i\frac{\p p_1}{\p x_j};
$$
thus, there exists $\lambda\in\bb{C}$ such that
$$
\sum_{i,j=0}^na_{i,j}x_i\frac{\p p_1}{\p x_j}=\lambda\cdot p_1,
$$
i.e., by Euler's formula,
$$
\sum_{i,j=0}^n(a_{i,j}-\frac{\lambda}{d_1}\delta_{i,j})x_i\frac{\p p_1}{\p x_j}=0.
$$
Since $d_1\geq 3$ and $p_1$ is smooth, $x_i\frac{\p p_1}{\p x_j}, i,j=0,\cdots, n$ are linearly independent over $\bb{C}$. Therefore,
$$
a_{i,j}=\frac{\lambda}{d_1}\delta_{i,j},\qquad i,j=0,\cdots, n.
$$
It follows that
$$
g=\frac{d}{d_1}\cdot\lambda\cdot f.
$$
Consequently, $h$ is a multiple of $f$, we are done.
\end{proof}

So we can construct a totally tangentially unstable polynomial of degree $d\geq 7$ by choosing arbitrarily two smooth homogeneous polynomial of different degree and of degree $\geq 3$. For instance, $f=(x_0^3+\cdots+x_n^3)(x_0^{d-3}+\cdots+x_n^{d-3})$ is such a good polynomial.

\section{Application of totally tangential instability}
In this section, we shall discuss some consequences of generic totally tangential instability. In the sequel, we denote $Or(f)=\bb{P}(G\cdot f)$ and $T_fOr(f)$ will be its projective tangent space at $f$.

\subsection{Positive dimensional linear subspaces contained in the tangent spaces to orbits}

\begin{lem}\label{lem: germTTI}
	Given $f\in \bb{P}(S_{n,d})$. Then $f$ is totally tangentially unstable if and only if the germ $(Or(f),f)\subseteq(\bb{P}(S_{n,d}),f)$ does not contain any germ of linear subspace $(E,f)\subseteq(\bb{P}(S_{n,d}),f)$ with $\dim E>0$.
\end{lem}

\begin{proof}
	Let $h\in\bb{T}_f$, then the germ of the line
	$$
	E_h=\{f+th: t\in\bb{C}\}\subseteq Or(f)
	$$
	is a linear subspace passing through $f$. Moreover, any linear subspace $E\subseteq Or(f)$ of positive dimension passing through $f$ contains a line $E_h$ for some $h\in\bb{T}_f-\{f\}$.

    If $f$ is totally tangentially unstable, then $\bb{T}_f=\{f\}$, so there are no germs of linear subspace $(E,f)\subseteq (Or(f),f)$ satisfying $\dim E>0$.
	Conversely, suppose there is no germs of linear subspace of $(Or(f),f)$ of positive dimension, then for any $h\in \bb{T}_f$, $(E_h,f))\subseteq(Or(f),f)$ has dimension 0, i.e., $E_h=\{f\}$, so $h=f$ in $\bb{P}(S_{n,d})$ and thus $f$ is totally tangentially unstable.
\end{proof}

Note that the linear space $E$ in Lemma \ref{lem: germTTI} is a linear subspace of $T_fOr(f)=\bb{P}(J_{f,d})$. This motivates us to consider the set $T_fOr(f)\cap Or(f)$. A similar argument gives the following.

\begin{cor}
	Let $f\in\bb{P}(S_{n,d})$. Then $f$ is totally tangentially unstable if and only if the germ $(T_fOr(f)\cap Or(f), f)$ does not contain any line (germs at $f$).
\end{cor}

By Theorem \ref{thm: TTI}, we get

\begin{cor}
Let $n\geq 1$ and $d\geq 4$. Then for a general $f\in\bb{P}(S_{n,d})$, $(T_fOr(f)\cap Or(f), f)$ does not contain a linear subspace of positive dimension.
\end{cor}

\begin{rk}
We also have a global version of the above result: if $n\geq1$ and $d\geq 4$, then for a general $f\in\bb{P}(S_{n,d})$, the closure of the orbit $Or(f)$ does not contain any projective line.
\end{rk}

\subsection{Local variation of tangent spaces to orbits}

Now let $n\geq 2, d\geq4, N=(n+1)^2-1$ and $\mathcal{O}\subseteq\bb{P}(S_{n,d})$ be the set of smooth polynomials which are not of Sebastiani-Thom type. As is shown in \cite{ZW}, the map
$$
\varphi:\mathcal{O}\longrightarrow \text{Grass}(N,\bb{P}(S_{n,d}))
$$
$$
f\mapsto T_fOr(f)=\bb{P}(J_{f,d})
$$
is injective.

Applying the totally tangential instability, we can now give a more precise description of the map $\varphi$, essentially showing that the Jacobian ideal varies severely even locally.

\begin{prop}\label{prop: variation}
	Let $n\geq 3, d\geq 4$ except $(n,d)=(3,4)$, then for any $f\in\mathcal{O}$ and any open neighborhood $\mathcal{N}_f\subseteq\mathcal{O}$ of $f$ in the strong topology,
	$$
	\bigcap_{g\in\mathcal{N}_f}\varphi(g)=\emptyset.
	$$
In particular, there exist $N+2=(n+1)^2+1$ elements $g_1,g_2,\cdots, g_{N+2}$ in $\mathcal{N}_f$ such that $\bigcap_{n=1}^{N+2}\varphi(g_i)=\emptyset$.
\end{prop}

\begin{proof}
For the first assertion, it is equivalent to show $\bigcap_{g\in\mathcal{N}_f}J_{g,d}=\{0\}$.

Indeed, let $h$ belong to the intersection, then for any \emph{general} $g\in\mathcal{N}_f$, we have that $h\in J_{g+th,d}$ for any small $t$.

We claim that $h\in\mathcal{T}_g$. Indeed, $h\in J_{g+th}$ implies that $\{H_{g+th}\}$ gives a deformation of $H_g$ that is infinitesimally trivial at any $t$. It follows that this family is locally trivial by the theorem in \cite{KK}, p.199. Namely, for $t$ small, $H_{g+th}$ is isomorphic to $H_g$. By the theorem in \cite{MM}, we obtain that $g+th\in G\cdot g$ for $t$ small. The claim follows by the definition of $\mathcal{T}_g$.

Further, $g$ is totally tangentially unstable by Theorem \ref{thm: TTI}, so $\mathcal{T}_g=\bb{C} g$. It follows that for any $g_1\neq g_2$ in $\mathcal{N}_f$ generically chosen,
$$
h\in\mathcal{T}_{g_1}\cap\mathcal{T}_{g_2}=\bb{C} g_1\cap\bb{C} g_2=\{0\}.
$$

For the last statement, observe that $\varphi(g)=\bb{P}(J_{g,d})$ is a linear subspace in $\bb{P}(S_{n,d})$, hence so is any intersection of finitely many $\varphi(g)$'s.

First choose any $g_1\in\mathcal{N}_f$, then $\dim\varphi(g_1)\leq N=(n+1)^2-1$. Since $\bigcap_{g\in\mathcal{N}_f}\varphi(g)=\emptyset$, we can choose $g_2\in\mathcal{N}_f$ such that $\dim\varphi(g_1)\cap\varphi(g_2)\leq N-1$. Inductively, one can choose $g_3,\cdots, g_{N+2}$ such that
$$
\dim\biggl(\varphi(g_1)\cap\varphi(g_2)\cap\cdots\cap\varphi(g_i)\biggr)\leq N-i+1,
$$
In particular, $\varphi(g_1)\cap\varphi(g_2)\cap\cdots\cap\varphi(g_{N+2})=\emptyset$.
\end{proof}

\subsection{General divisors and isotrivial linear systems}

A family is a flat morphism $f: X\to B$ of complex varieties with connected fibers. The family is called {\bf isotrivial} if there exists a Zariski open dense subset $U\subseteq B$ such that $f^{-1}(x)$ and $f^{-1}(y)$ are isomorphic for any $(x,y)\in U\times U$.

Let $n\geq 2$ and $d\geq 1$. Let $\mathscr{L}$ be a linear system of divisors of degree $d$ on $\bb{P}^n$, then $\mathscr{L}$ can be naturally seen as a family. We call $\mathscr{L}$ isotrivial if the associated family is isotrivial.

Now Let $n\geq 3$ and $d\geq 4$ with $(n,d)=(3,4)$ excluded. Given a pencil of degree $d$ divisor on $\bb{P}^n$, say $\{H_{\lambda f+\mu h}\}$ with $f,h\in S_{n,d}$. If this pencil is isotrivial, then we may assume that a general element is isomorphic to $H_f$. If $f$ is generically chosen, then as in the proof of Proposition \ref{prop: variation}, we get that $h\in\mathcal{T}_f=\bb{C}f$. Hence the pencil $\{H_{\lambda f+\mu h}\}$ degenerates to a single divisor $\{H_f\}$. In fact, we have proved the following.

\begin{prop}
Let $n\geq 3, d\geq 4$ with $(n,d)=(3,4)$ excluded. Then a general divisor of degree $d$ on $\bb{P}^n$ does not belong to an isotrivial linear system of positive dimension.
\end{prop}

\section{Tangential smoothability for singular polynomials}

 Let $n\geq 2, d\geq2$ and $\mathcal{S}\subseteq\bb{P}(S_{n,d})$ be the set of singular polynomials. Then $\mathcal{S}$ is an irreducible hypersurface, see for instance \cite{VO2}, Section 2.1. Write a general $f\in\bb{P}(S_{n,d})$ as
$$
f=\sum\limits_{|\alpha|=d} c_\alpha x^\alpha,
$$
then $(c_\alpha: |\alpha|=d)$ can be regarded as the homogeneous coordinates of $\bb{P}(S_{n,d})$. Denote by $P_{s}(f):=P_s(c_\alpha)=0$ the defining equation of $\mathcal{S}$ in $\bb{P}(S_{n,d})$.

Now given $f\in S_{n,d}$,
$$
f=\sum_{|\alpha|=d}c_\alpha x^\alpha,\qquad c_\alpha\in\bb{C}
$$
and $h\in J_{f,d}$
$$
h=\sum_{\beta,\gamma=0}^n a_{\beta,\gamma}x_\beta \frac{\p f}{\p x_\gamma},
$$
$P_s(f+h)$ is a polynomial in $c_\alpha,|\alpha|=d$ and $a_{\beta,\gamma},\beta,\gamma=0,\cdots, n$. In particular, for fixed $c_\alpha$ or equivalently $f\in S_{n,d}$, $P_s(f+h)$ is a polynomial in $a_{\beta,\gamma},\beta,\gamma=0,\cdots, n$. We shall also use the notation $P_s(c_\alpha,a_{\beta,\gamma}):=P_s(f,a_{\beta,\gamma}):=P_s(f+h)$.

Consider $J_{f,d}$ as a vector subspace of $\bb{C}^{(n+1)^2}$ with affine coordinates $(a_{\beta,\gamma}:\beta,\gamma=0,\cdots,n)$. Then, $f\in\mathcal{S}$ is tangentially smoothable (see also Definition \ref{def: TS}) if and only if restricted to $J_{f,d}$, $P_s(f,a_{\beta,\gamma})\neq 0$ as a polynomial in $a_{\beta,\gamma}$'s. Observe that $J_{f,d}$ may not have dimension $(n+1)^2$ since $\{x_\beta \frac{\p f}{\p x_\gamma}:\beta,\gamma=0,\cdots, n\}$ can be linearly dependent over $\bb{C}$.

Immediately, our discussion above gives the following.

\begin{lem}
	Given $f\in\mathcal{S}$ tangentially smoothable, there exists a nonempty Zariski open $JU_f\subseteq J_{f,d}$ such that for any $h\in JU_f$, $f+h$ is smooth.
\end{lem}

The following Proposition shows in particular that not all $f\in\mathcal{S}$ are tangentially smoothable.

\begin{prop}\label{prop: nonexistence}
	If $f\in S_{n,d}$ and $H_f\subseteq\bb{P}^n$ admits a singularity $p$ of multiplicity $\geq 3$, then $f$ is not tangentially smoothable, i.e., there exists no $h\in J_{f,d}$ so that $f+h$ is smooth.
\end{prop}

\begin{proof}
	Without loss of generality we may assume $p=(0:\cdots: 0:1)$. Then $f$ has the following form
	$$
	f(x_0,\cdots,x_n)=v_3(x_0,\cdots,x_{n-1})x_n^{d-3}+\cdots+v_d(x_0,\cdots,x_{n-1}),
	$$
	where $v_j\in\bb{C}[x_0,\cdots, x_{n-1}]_j, j=3,\cdots, n$.
	A fortiori, $f\in (x_0,\cdots, x_{n-1})^3$ and so
$$
\frac{\p f}{\p x_i}\in (x_0,\cdots, x_{n-1})^2,\qquad i=0,\cdots, n.
$$
 Therefore for any $h\in J_{f,d}$, $h\in (x_0,\cdots, x_{n-1})^2$ and thus $f+h\in (x_0,\cdots, x_{n-1})^2$, which implies that $f+h$ has $p=(0:\cdots: 0:1)$ as a singular point.\qedhere
\end{proof}

\subsection{Proof of Theorem \ref{thm: TS} }

We are now to prove the genericity of tangentially smoothability.

The ``only if" part follows from Proposition \ref{prop: nonexistence}, so we focus on the other part.

Suppose $f$ is given such that $H_f$ has only isolated singularities and every singular point has multiplicity 2. Let $p_1,\cdots, p_m$ be all the singular points of $H_f$. We have

\begin{claim}\label{claim: genericTS}
{\rm (i)} For a general element $h\in J_{f,d}$, $H_h$ is smooth away from $p_1,\cdots,p_m$;

{\rm (ii)} For any $j\in\{1,\cdots, m\}$, there exists an $h_j\in J_{f,d}$ such that $V(h_j)$ is smooth at $p_j$.
\end{claim}

Assuming the claim, by (i), there exists a nonempty Zariski open subset, say $JU_{f,0}\subseteq J_{f,d}$ such that for any $h\in JU_{f,0}$, $H_h$ is smooth away from $p_1,\cdots, p_m$. Moreover, for any $j\in\{1,\cdots, m\}$, define
$$
JU_{f,j}=\{h\in J_{f,d}\quad:\quad H_h\text{ is smooth at } p_j \}.
$$
Since any $h\in J_{f,d}$ can be written as $h=\sum\limits_{\beta,\gamma=0}^{n}a_{\beta,\gamma}x_\beta \frac{\p f}{\p x_\gamma}$, $H_h$ has $p_j$ as a singularity if and only if $\nabla h(p_j)=0$ which is also equivalent to a system of polynomial equations in $a_{\beta,\gamma}$'s. Therefore, $JU_{f,j}$ is a Zariski open subset of $J_{f,d}$, and it is nonempty by (ii).

Let
$$
JU_f=JU_{f,0}\cap JU_{f,1}\cap \cdots\cap JU_{f,m},
$$
then it is a nonempty Zariski open subset of $J_{f,d}$. Moreover, for any $h\in JU_f$, we have
\begin{enumerate}
\item $H_h$ is smooth away from $p_1,\cdots,p_m$, since $h\in JU_{f,0}$;

\item For any $j$,  $H_h$ is smooth at $p_j$, since $h\in JU_{f,j}$.
\end{enumerate}
So $h$ is smooth. Obviously $h=f+(h-f)$ and $h-f\in J_{f,d}$, therefore $f$ is tangentially smoothable by definition.

{\it Proof of Claim \ref{claim: genericTS}: }
For (i),  we may see the vector space $J_{f,d}$ as a linear system over $\bb{P}^n$, then clearly the base point set of this linear system equals $\{p_1,\cdots, p_m\}$, then (i) follows from Bertini theorem.

For (ii), we may assume that $p_j=(0:\cdots:0:1)$ and thus $f$ has the following form
$$
f=(x_0^2+\cdots+x_{n-1}^2)x_{n}^{d-2}+v_3(x_0,\cdots,x_{n-1})x_n^{d-3}+\cdots+v_d(x_0,\cdots,x_{n-1}),
$$
for $k\geq 0$ where $v_j\in\bb{C}[x_0,\cdots, x_{n-1}]_j, j=3,\cdots, n$. Then we have
$$
J_{f,d}\ni x_n\frac{\p f}{\p x_0}=2x_0x_n^{d-1}+\frac{\partial v_3}{\partial x_0}x_n^{d-2}+\cdots+\frac{\partial v_d}{\partial x_0}x_n.
$$
Set $h_j=f+x_n\frac{\p f}{\p x_0}\in J_{f,d}$, then $h_j$ does not have $p_j=(0:\cdots:0:1)$ as a singularity, so (ii) holds.\qed

\end{document}